\documentclass[11pt]{article}

\usepackage{amssymb,amsmath,amsfonts,amsthm}
\usepackage{latexsym}
\usepackage{graphics}
\usepackage{indentfirst}
\usepackage{lmodern}
\usepackage{hyperref}
\usepackage{mathrsfs}
\usepackage{array}
\usepackage{mathtools}
\usepackage{graphicx}
\usepackage{setspace}
\usepackage{enumerate}
\usepackage{tikz}
\usepackage{url,color}
\usepackage{chngcntr}
\usepackage{multicol}
\usepackage{comment}

\newtheorem{innercustomthm}{Theorem}
\newenvironment{customthm}[1]
  {\renewcommand\theinnercustomthm{#1}\innercustomthm}
  {\endinnercustomthm}
\newtheorem{innercustomlem}{Lemma}
\newenvironment{customlem}[1]
  {\renewcommand\theinnercustomlem{#1}\innercustomlem}
  {\endinnercustomlem}
\newtheorem{innercustomcor}{Corollary}
\newenvironment{customcor}[1]
  {\renewcommand\theinnercustomcor{#1}\innercustomcor}
  {\endinnercustomcor}

\newtheorem*{thm*}{Theorem}
\newtheorem{thm}{Theorem}
\newtheorem{lem}[thm]{Lemma}
\newtheorem{pro}[thm]{Proposition}
\newtheorem{obs}[thm]{Observation}
\newtheorem{cor}[thm]{Corollary}

\newtheorem{ques}[thm]{Question}

\newcommand{\N}{\mathbb{N}}

\newcommand{\R}{\mathbb{R}}

\begin{document}

\title{Counting Packings of List-colorings of Graphs}

\author{Hemanshu Kaul\thanks{Department of Applied Mathematics, Illinois Institute of Technology, Chicago, IL 60616. E-mail: {\tt {kaul@iit.edu}}} \and Jeffrey A. Mudrock\thanks{Department of Mathematics and Statistics, University of South Alabama, Mobile, AL 36688.  E-mail:  {\tt {mudrock@southalabama.edu}}}}

\maketitle

\begin{abstract}
Given a list assignment for a graph, list packing asks for the existence of multiple pairwise disjoint list colorings of the graph. Several papers have recently appeared that study the existence of such a packing of list colorings. Formally, a proper $L$-packing of size $k$ of a graph $G$ is a set of $k$ pairwise disjoint proper $L$-colorings of $G$ where $L$ is a list assignment of colors to the vertices of $G$. In this note, we initiate the study of counting such packings of list colorings of a graph.
We define $P_\ell^\star(G,q,k)$ as the guaranteed number of proper $L$-packings of size $k$ of $G$ over all list assignments $L$ that assign $q$ colors to each vertex of $G$, and we let $P^\star(G,q,k)$ be its classical coloring counterpart. We let $P_\ell^\star(G,q)= P_\ell^\star(G,q,q)$ so that $P_\ell^\star(G,q)$ is the enumerative function for the previously studied list packing number $\chi_\ell^\star(G)$. Note that the chromatic polynomial of $G$, $P(G,q)$, is $P^\star(G,q,1)$, and the list color function of $G$, $P_\ell(G,q)$, is $P_\ell^\star(G,q,1)$.

Inspired by the well-known behavior of the list color function and the chromatic polynomial, 
we make progress towards the question of whether $P_{\ell}^\star(G,q,k) = P^\star(G,q,k)$ when $q$ is large enough. Our result generalizes the recent theorem of Dong and Zhang (2023), which improved results going back to Donner (1992), about when the list color function equals the chromatic polynomial. Further, we use a polynomial method to generalize bounds on the list packing number, $\chi_\ell^\star(G)$, of sparse graphs to exponential lower bounds (in the number of vertices of $G$) on the corresponding list packing functions, $P_\ell^\star(G,q)$.

\medskip



\end{abstract}

\section{Introduction}\label{intro}

List coloring is a fundamental topic in graph theory with a rich history since its introduction in 1970s~\cite{ET79, V76}. The basic problem studied under list coloring is the question of the existence of such a coloring, with Thomassen's 5-choosability of planar graphs~\cite{T94} being a prime example of such a result. In the case of planar graphs, the original setting for classical coloring problems, there are also the corresponding, well-studied questions (see the discussions in ~\cite{CC23, DK23}) of the existence of exponentially many list colorings in the number of vertices of the graph. Recently, another perspective on existence of many list colorings was given by Cambie, Batenburg, Davies, and Kang~\cite{CC21}. They asked for the existence of simultaneous pairwise disjoint list colorings of a graph which they called a packing of list colorings. A proper $L$-packing of size $k$ of $G$ is a set of $k$ pairwise disjoint proper $L$-colorings of $G$ where $L$ is a list assignment of colors to the vertices of $G$.

In this paper, we initiate the study of counting such packings of list colorings of a graph. We define the $(k,q)$-fold list packing function, $P_\ell^\star(G,q,k)$ as the guaranteed number of proper $L$-packings of size $k$ of $G$ over all list assignments $L$ that assign $q$ colors to each vertex of $G$. Similarly, the $(k,q)$-fold classical packing function, $P^\star(G,q,k)$ is the number of proper $L$-packings of size $k$ of $G$ for the constant list assignment $L$ that assigns $\{1, \ldots, q \}$ to each vertex. Various previously defined and studied parameters can now be defined in terms of these functions. The list packing number $\chi_\ell^\star(G)$ is the least $q$ such that $P_\ell^\star(G,q,q)>0$. In fact, the case $k=q$ is of special interest, and we define $P_\ell^\star(G,q)$ to be the guaranteed number of proper $L$-packings of size $q$ of $G$ over all list assignments $L$ that assign $q$ colors to each vertex of $G$. The chromatic polynomial of $G$, $P(G,q)$, is $P^\star(G,q,1)$, and the list color function of $G$, $P_\ell(G,q)$, is $P_\ell^\star(G,q,1)$. 

It is a fundamental question to ask how the enumerative function of a new notion of coloring compares to the classical chromatic polynomial. It is well known that the list color function need not always equal the chromatic polynomial, but it does equal the chromatic polynomial when the number of colors is large enough. Recently, Dong and Zhang~\cite{DZ22} (improving upon results in~\cite{D92}, \cite{T09}, and~\cite{WQ17}, that answered a question of Kostochka and Sidorenko~\cite{KS90}) showed that $P_{\ell}(G,q)=P(G,q)$ whenever $q \geq |E(G)|-1$ for any graph $G$. Naturally, we ask whether $P_{\ell}^\star(G,q,k)$ equals $P^\star(G,q,k)$ when $q$ is large enough. Towards this, we show that $P_{\ell}^\star(G,q,k) = P^\star(G,q,k)$ whenever $q \geq nk(k-1)/2 + mk - 1$ where $G$ is an $n$-vertex graph with $m$ edges. This generalizes the aforementioned result of Dong and Zhang which corresponds to the case $k=1$. We also affirmatively answer this question for trees when $k=q$.

Recently, two sets of authors~\cite{CC23, CS24} have obtained new bounds on the list packing number of planar graphs and its subfamilies. Recall, $\chi_\ell^\star(G) \le q$ is equivalent to $P_\ell^\star(G,q)>0$. We use a polynomial method (see the discussion in~\cite{T14}), previously used in counting list colorings, DP-colorings, and colorings of $S$-labeled graphs~\cite{BG22,DKM22,DK23}, to generalize bounds on the list packing number of sparse graphs to exponential lower bounds on the corresponding list packing functions. 

In the rest of this section, we formally define these concepts and give an outline of our results.

\subsection{Basic Terminology and Notation}

In this paper all graphs are nonempty, finite, simple graphs unless otherwise noted.  Generally speaking we follow West~\cite{W01} for terminology and notation.  The set of natural numbers is $\N = \{1,2,3, \ldots \}$.  For $m \in \N$, we write $[m]$ for the set $\{1, \ldots, m \}$, and we take $[0]$ to be the empty set.  We write $!m$ for the number of derangements of $[m]$. Also, $K_{n,m}$ denotes the complete bipartite graphs with partite sets of size $n$ and $m$.  

If $G$ is a graph and $S \subseteq V(G)$, we use $G[S]$ for the subgraph of $G$ induced by $S$.  If $u$ and $v$ are adjacent in $G$, $uv$ or $vu$ refers to the edge between $u$ and $v$.  We write $N_G(v)$ (resp., $N_G[v]$) for the neighborhood (resp., closed neighborhood) of vertex $v$ in the graph $G$. The maximum average degree of a graph $G$, denoted $\text{mad}(G)$, is the maximum of the average degrees of its subgraphs.

The Cartesian product of graphs $G$ and $H$, denoted $G \mathbin{\square} H$, is the graph with vertex set $V(G) \times V(H)$ and edges created so that $(u,v)$ is adjacent to $(u',v')$ if and only if either $u=u'$ and $vv' \in E(H)$ or $v=v'$ and $uu' \in E(G)$.

\subsection{List Coloring and the List Color Function} 

In classical vertex coloring one wishes to color the vertices of a graph $G$ with colors from $[q]$ so that adjacent vertices in $G$ receive different colors, a so-called \emph{proper $q$-coloring}.  The \emph{chromatic number} of a graph, denoted $\chi(G)$, is the smallest $q$ such that $G$ has a proper $q$-coloring.  List coloring is a generalization of classical vertex coloring introduced independently by Vizing~\cite{V76} and Erd\H{o}s, Rubin, and Taylor~\cite{ET79} in the 1970s.  In list coloring, we associate a \emph{list assignment} $L$ with a graph $G$ so that each vertex $v \in V(G)$ is assigned a list of available colors $L(v)$ (we say $L$ is a list assignment for $G$).  We say $G$ is \emph{$L$-colorable} if there is a proper coloring $f$ of $G$ such that $f(v) \in L(v)$ for each $v \in V(G)$ (we refer to $f$ as a \emph{proper $L$-coloring} of $G$).  A list assignment $L$ is called a \emph{$q$-assignment} for $G$ if $|L(v)|=q$ for each $v \in V(G)$.  We say $G$ is \emph{$q$-choosable} if $G$ is $L$-colorable whenever $L$ is a $q$-assignment for $G$.  The \emph{list chromatic number} of a graph $G$, denoted $\chi_\ell(G)$, is the smallest $q$ such that $G$ is $q$-choosable.  It is immediately obvious that for any graph $G$, $\chi(G) \leq \chi_\ell(G)$.  Moreover, it is well-known that the gap between the chromatic number and list chromatic number of a graph can be arbitrarily large.

In 1912 Birkhoff~\cite{B12} introduced the notion of the chromatic polynomial with the hope of using it to make progress on the four color problem.  For $q \in \N$, the \emph{chromatic polynomial} of a graph $G$, $P(G,q)$, is the number of proper $q$-colorings of $G$. It is well-known that $P(G,q)$ is a polynomial in $q$ of degree $|V(G)|$ (e.g., see~\cite{DKT05, B94}). 
%For example, $P(K_n,m) = \prod_{i=0}^{n-1} (m-i)$, $P(C_n,m) = (m-1)^n + (-1)^n (m-1)$, $P(T,m) = m(m-1)^{n-1}$ whenever $T$ is a tree on $n$ vertices, and $P(K_{2,l},m) = m(m-1)^l + m(m-1)(m-2)^l$ (see~\cite{B94} and~\cite{W01}). 

The notion of chromatic polynomial was extended to list coloring in the early 1990s by Kostochka and Sidorenko~\cite{KS90}.  If $L$ is a list assignment for $G$, we use $P(G,L)$ to denote the number of proper $L$-colorings of $G$. The \emph{list color function} $P_\ell(G,q)$ is the minimum value of $P(G,L)$ where the minimum is taken over all possible $q$-assignments $L$ for $G$.  Since a $q$-assignment could assign the same $q$ colors to every vertex in a graph, it is clear that $P_\ell(G,q) \leq P(G,q)$ for each $q \in \N$.  In general, the list color function can differ significantly from the chromatic polynomial for small values of $q$.  So naturally understanding the list color functions of graphs for small values of $q$ has received some attention in the literature (see e.g.,~\cite{T09} for some discussion).  Recently, a bound on the list color functions of sufficiently sparse graphs was discovered using a polynomial technique, and this bound is of particular interest for small values of $q$.

\begin{pro}[\cite{DK23}]\label{prop: listcoloring}
Suppose $G$ is an $n$-vertex graph with $m$ edges, and $q$ is a positive integer greater than 1 satisfying $\chi_{\ell}(G) \leq q$.  If $m \leq (q-1)n$, then 
$$P_{\ell}(G,q) \geq q^{n-\frac{m}{q-1}}.$$
\end{pro}

On the other hand, in 1992, answering a question of Kostochka and Sidorenko~\cite{KS90}, Donner~\cite{D92} showed that for any graph $G$ there is an $N \in \N$ such that $P_\ell(G,q) = P(G,q)$ whenever $q \geq N$. Dong and Zhang~\cite{DZ22} (improving upon results in~\cite{D92}, \cite{T09}, and~\cite{WQ17}) subsequently showed the following.

\begin{thm} \label{thm: DZ}
For any graph $G$, $P_{\ell}(G,q)=P(G,q)$ whenever $q \geq |E(G)|-1$.
\end{thm}

\subsection{List Packing and the List Packing Function}

We begin with some definitions related to list packing (we follow~\cite{CC21}).  Suppose $L$ is a list assignment for graph $G$.  An \emph{$L$-packing of size $k$ of $G$} is a set of $k$ $L$-colorings of $G$, $\{f_1, \ldots, f_k \}$, such that $f_i(v) \neq f_j(v)$ whenever $i, j \in [k]$, $i \neq j$, and $v \in V(G)$.  Moreover, we say that $\{f_1, \ldots, f_k \}$ is \emph{proper} if $f_i$ is a proper $L$-coloring of $G$ for each $i \in [k]$.  It can be shown that for any graph $G$, there is an $m \in \N$ such that $G$ has a proper $L$-packing of size $m$ whenever $L$ is an $m$-assignment for $G$ (see e.g.,~\cite{M23}). The \emph{list packing number} of $G$, $\chi_{\ell}^\star(G)$, is the least $k$ such that $G$ has a proper $L$-packing of size $k$ whenever $L$ is a $k$-assignment for $G$.

Suppose that $L$ is a list assignment for graph $G$.  Let $P^\star(G,L,k)$ denote the number of proper $L$-packings of size $k$ of $G$.  The \emph{list packing function} of $G$, denoted $P_{\ell}^\star(G,q)$, is the minimum value of $P^\star(G,L,q)$ taken over all $q$-assignments $L$ for $G$.  The \emph{classical packing function} of $G$, denoted $P^\star(G,q)$, is equal to $P^\star(G,L,q)$ where $L$ is the list assignment that assigns $[q]$ to each vertex in $V(G)$.  Clearly, $P_{\ell}^\star(G,q) \leq P^\star(G,q)$ for each $q \in \N$. 

More generally, for each $k,q \in \N$ with $k \leq q$ we also define the \emph{$(k,q)$-fold list packing function} of $G$, denoted $P_{\ell}^{\star}(G,q,k)$, as the minimum value of $P^\star(G,L,k)$ taken over all $q$-assignments $L$ for $G$.  Also, the \emph{$(k,q)$-fold classical packing function} of $G$, denoted $P^\star(G,q,k)$, is equal to $P^\star(G,L,k)$ where $L$ is the list assignment that assigns $[q]$ to each vertex in $V(G)$.  Clearly, $P_{\ell}^\star(G,q,k) \leq P^\star(G,q,k)$ for each $q,k \in \N$ satisfying $k \leq q$.  Notice that based upon these definitions, for $q \in \N$, $P_{\ell}^{\star}(G,q,q)= P_{\ell}^{\star}(G,q)$, $P^{\star}(G,q,q)= P^{\star}(G,q)$, $P_{\ell}^{\star}(G,q,1)= P_{\ell}(G,q)$, and $P^{\star}(G,q,1)= P(G,q)$.  

\subsection{Outline of Results}

With Donner's aforementioned result (that answered Kostochka and Sidorenko's corresponding question) in mind, the following question is natural.

\begin{ques} \label{ques: asy}
For every graph $G$ does there exist an $N \in \N$ such that $P_{\ell}^\star(G,q,k) = P^\star(G,q,k)$ whenever $k \leq q$ and $q \geq N$? The case $k=q$ is of particular interest.
\end{ques}

By Donner's result, the answer to Question~\ref{ques: asy} becomes yes when $k \leq q$ is replaced with $k = 1$.

In Section~\ref{trees}, we answer Question~\ref{ques: asy} for trees in the case $k=q$.

\begin{thm} \label{pro: tree}
If $T$ is a tree on $n$ vertices and $q \in \N$, then $P_{\ell}^\star(T,q) = P^{\star}(T,q) = (!q)^{n-1} .$ 
\end{thm}

In Section~\ref{Cartesian}, we make further progress on Question~\ref{ques: asy} by making a connection to the Cartesian product of graphs using the framework established in \cite{M23}. One consequence of this connection is the explicit reformulation of the classical packing function in terms of the chromatic polynomial: $P^\star(G,q,k) = P(G \mathbin{\square} K_k, q)/{k!}$, which is equivalent to: $P^\star(G,q,k) = {L(n,k,q)}/{k!}$, where $L(n,k,q)$ is the number of $n \times k$ Latin arrays containing at most $q$ symbols. 

We prove the following generalization of Theorem~\ref{thm: DZ} from the context of counting list colorings to counting packings of list colorings.

\begin{thm} \label{thm: equal}
Suppose $G$ is an $n$-vertex graph with $m$ edges.  If $q,k \in \N$ satisfy $q \geq nk(k-1)/2 + mk - 1$, then  $P_{\ell}^\star(G,q,k) = P^\star(G,q,k)$.
\end{thm}

Note that Theorem~\ref{thm: DZ} corresponds to the case $k=1$ of Theorem~\ref{thm: equal}. 
 Theorem~\ref{thm: equal} makes partial progress towards Question~\ref{ques: asy}, but note that it requires $q$ to be at least quadratic in $k$. As a next step, it would be meaningful to improve this bound on $q$ to a linear function of $k$. 

In Section~\ref{exp}, we use the framework from Section~\ref{Cartesian} and a simplified version (from~\cite{BG22}) of a well-known result of Alon and F\"{u}redi~\cite{AF93} on the number of non-zeros of a polynomial to generalize the bounds on the list packing number to their enumerative counterparts, leading to exponential lower bounds on the corresponding list packing functions of sparse graphs.

\begin{lem}\label{lem: together}
Suppose $G$ is an $n$-vertex graph with $m$ edges.  Suppose $L$ is a $q$-assignment for $G$, and $k \in \N$ is such that $k \leq q$ and $P^\star(G,L,k) > 0$.  If $m \leq n(q-1-(k-1)/2)$, then 
$$P^{\star}(G,L,k) \geq \frac{1}{k!} q^{kn-\frac{nk(k-1)/2 + km}{q-1}}.$$
\end{lem}

It is natural to compare Lemma~\ref{lem: together}, which counts packings of list colorings, with Proposition~\ref{prop: listcoloring} which counts list colorings. When we plug in $k=1$ in Lemma~\ref{lem: together}, we get the lower bound on the number of list colorings, $q^{n-\frac{m}{q-1}}$, that is implied by Proposition~\ref{prop: listcoloring}, and the required bound on the number of edges, $m \leq n(q-1)$, also remains the same. 

As an illustration, we combine a result from~\cite{CC23} and Lemma~\ref{lem: together} to show there are exponentially many pairs of disjoint $L$-colorings for any $3$-assignment $L$ of a planar graph of girth at least 8.  

\begin{cor} \label{cor: 2colorings}
Suppose $G$ is an $n$-vertex planar graph of girth at least 8.  Then,
$$P_{\ell}^\star(G,3,2) \geq \frac{3^{n/6}}{2}.$$
\end{cor}

Note that it is shown in \cite{CC23} that there is an $n$-vertex planar graph $G$ of girth at least 8 satisfying $P_{\ell}^\star(G,3,3)=0$.

Moreover, any future  improvement in the bounds on the list packing number of sparse graphs (as defined by the bound on number of edges in  Lemma~\ref{lem: together}) would immediately lead to corresponding exponentially many list packings using Lemma~\ref{lem: together} without any additional work.

\section{Trees} \label{trees}

In this section, our aim is prove the following which gives an affirmative answer to Question~\ref{ques: asy} for trees when $k=q$.

\begin{customthm} {\bf \ref{pro: tree}}
If $T$ is a tree on $n$ vertices and $q \in \N$, then $P_{\ell}^\star(T,q) = P^{\star}(T,q) = (!q)^{n-1} .$ 
\end{customthm}

We start with a couple of trivial trees.

\begin{pro} \label{pro: easytree}
  $P^\star_{\ell}(K_1,q)  = P^\star(K_1,q) = 1$ and $P^\star(K_2, q) = !q$.
\end{pro} 

\begin{proof}
Suppose $G_i = K_i$ for $i \in [2]$ and that $L_i$ is the $q$-assignment for $G_i$ that assigns the list $[q]$ to each vertex of $G_i$.  The result is obvious in the case that $i=1$.  So, suppose that $i=2$ and $V(G_2) = \{x,y \}$.  In the case that $q=1$ it is clear that $P^\star(G_2, L_2, 1) = 0 = !1$ since there is no proper $L_2$-coloring of $G_2$.  Now, suppose that $q \geq 2$.  Let $A$ be the set of all proper $L_2$-packings of size $q$ of $G_2$, and let $D_q$ be the set of all derangements of $[q]$.  Consider the function $f: D_q \rightarrow A$ given by $f(d) = \{f_1, \ldots, f_q \}$ where $f_i$ is the proper $L_2$-coloring of $G_2$ given by $f_i(x) = i$ and $f_i(y) = d(i)$ (clearly $\{f_1, \ldots, f_q \} \in A$).  It is easy to see that $f$ is a bijection.  Consequently, $P^\star(G_2, L_2, q) = |A| = |D_q| = !q$. 
\end{proof}  

Now, we can find the classical packing function for all trees.

\begin{pro} \label{pro: classictree}
Suppose $T$ is a tree on $n$ vertices.  Then, $P^\star(T,q) = (!q)^{n-1}$.
\end{pro}

\begin{proof}
The proof is by induction on $n$.  Note that when $n \in [2]$, the desired result follows by Proposition~\ref{pro: easytree}.  So, suppose that $n \geq 3$ and the desired result holds for all natural numbers less than $n$.  Suppose $L$ is the $q$-assignment for $T$ that assigns the list $[q]$ to each vertex of $T$.  In the case that $q=1$ it is clear that $P^\star(T, L, 1) = 0 = (!1)^{n-1}$ since there is no proper $L$-coloring of $T$.  Now, suppose that $q \geq 2$ and $y$ is a leaf of $T$ with $N_T(y) = \{x\}$.  Let $T' = T - y$.  Note that $T'$ is a tree on $n-1$ vertices, and let $L'$ be the $q$-assignment for $T'$ obtained by restricting the domain of $L$ to $V(T')$.  Let $A$ be the set of all proper $L'$-packings of size $q$ of $T'$.  By the inductive hypothesis we know $|A| = (!q)^{n-2}$.  Let $B$ be the set of all proper $L$-packings of size $q$ of $T$, and let $D_q$ be the set of all derangements of $[q]$.  Consider the function $f : A \times D_q \rightarrow B$ given by $f((\{f_1, \ldots, f_q \}, d)) = \{g_1, \ldots, g_q \}$ where $g_i$ is the proper $L$-coloring of $T$ given by 
\[g_i(v) = \begin{cases} 
      f_i(v) &\text{if $v \in V(T')$}  \\
      d(f_i(x)) &\text{if $v=y$.} 
   \end{cases}
\]
It is easy to see that $f$ is a bijection.  So, $P^\star(T,L,q) = |B| = |A| |D_q| = (!q)^{n-1}$. 
\end{proof}

We need one more lemma before we prove Theorem~\ref{pro: tree}.

\begin{lem} \label{lem: fixedpoints}
Suppose that $A$ and $B$ are sets of size $q \in \N$.  Then, the number of bijections from $A$ to $B$ with no fixed points is at least $!q$.
\end{lem}

\begin{proof}
Note that the result is obvious when $A=B$.  So, suppose that $A \neq B$.  Name the elements of the sets so that $A = \{a_1, \ldots, a_q \}$, $B = \{b_1, \ldots, b_q \}$, and there is a $t \in [q]$ such that $a_i = b_i$ whenever $i \in [t-1]$ and $\{a_i : i \geq t \} \cap \{b_i: i \geq t\} = \emptyset$.  Let $D_q$ be the set of all derangements of $[q]$, and let $F$ be the set of all bijections from $A$ to $B$ with no fixed points.  Consider the function $f : D_q \rightarrow F$ given by $f(d) = g$ where $g$ is the element of $F$ defined by $g(a_i) = b_{d(i)}$ for each $i \in [q]$.  Since $f$ is injective, $|F| \geq |D_q| = !q$.  
\end{proof}

We are now ready to finish the proof of Theorem~\ref{pro: tree}.

\begin{proof}
We need only show $P_{\ell}^\star(T,q) = (!q)^{n-1}.$ The proof is by induction on $n$.  Note that when $n =1$, the desired result follows by Proposition~\ref{pro: easytree}.  So, suppose that $n \geq 2$ and the desired result holds for all natural numbers less than $n$.  In the case that $q=1$ it is clear that $0 \leq P_{\ell}^\star(T, 1) \leq P^\star(T,1)  = 0 = (!1)^{n-1}$.  Now, suppose that $q \geq 2$, and $L$ is a $q$-assignment for $T$ such that $P^\star(T,L,q) = P^\star_{\ell}(T,q)$. Suppose $y$ is a leaf of $T$ with $N_T(y) = \{x\}$.  Let $T' = T - y$.  Note that $T'$ is a tree on $n-1$ vertices, and let $L'$ be the $q$-assignment for $T'$ obtained by restricting the domain of $L$ to $V(T')$.  Let $A$ be the set of all proper $L'$-packings of size $q$ of $T'$.  By the inductive hypothesis we know $|A| \geq (!q)^{n-2}$.  Let $B$ be the set of all proper $L$-packings of size $q$ of $T$.  Let $D$ be the set of all bijections from $L(x)$ to $L(y)$ without any fixed points.  Consider the function $f : A \times D \rightarrow B$ given by $f((\{f_1, \ldots, f_q \}, d)) = \{g_1, \ldots, g_q \}$ where $g_i$ is the proper $L$-coloring of $T$ given by 
\[g_i(v) = \begin{cases} 
      f_i(v) &\text{if $v \in V(T')$}  \\
      d(f_i(x)) &\text{if $v=y$.} 
   \end{cases}
\]
It is easy to see that $f$ is a bijection.  Using all we have deduced along with Proposition~\ref{pro: classictree} and Lemma~\ref{lem: fixedpoints} yields,
$$(!q)^{n-1} = P^\star(T,q) \geq P^\star_{\ell}(T,q) = P^\star(T,L,q) = |B| = |A| |D| \geq (!q) (!q)^{n-2}$$
as desired. 
\end{proof}

\section{A Connection to a Cartesian Product} \label{Cartesian}

In this section, we aim to prove the following result that generalizes Theorem~\ref{thm: DZ} from the context of counting list colorings to counting packings of list colorings. Our main idea will be a connection to the Cartesian product of a graph with a complete graph.

\begin{customthm} {\bf \ref{thm: equal}}
Suppose $G$ is an $n$-vertex graph with $m$ edges.  If $q,k \in \N$ satisfy $q \geq nk(k-1)/2 + mk - 1$, then  $P_{\ell}^\star(G,q,k) = P^\star(G,q,k)$.
\end{customthm}

In the rest of this section, suppose that $G$ is an $n$-vertex graph with $V(G) = \{v_1, \ldots, v_n \}$.  Additionally, when $H = G \mathbin{\square} K_k$ for some $k \in \N$, we will suppose that the vertex set of the copy of $K_k$ used to form $H$ is $\{w_1, \ldots, w_k \}$.  When $L$ is a $q$-assignment for $G$, we let $L^{(k)}$ be the $q$-assignment for $H = G \mathbin{\square} K_k$ given by $L^{(k)}(v,w_i) = L(v)$ for each $v \in V(G)$ and $i \in [k]$. 

The strategy in this section is to follow the framework established in~\cite{M23}. One key observation adapted from~\cite{M23} is as follows.

\begin{obs} \label{obs: key}
Suppose $L$ is a $q$-assignment for graph $G$. Graph $G$ has a proper $L$-packing of size $k$ if and only if there is a proper $L^{(k)}$-coloring of $G \mathbin{\square} K_k$.
\end{obs}

The following lemma reduces the classical packing function to the chromatic polynomial of the Cartesian product of the graph with a complete graph.

\begin{lem} \label{lem: connect}
For any graph $G$, and $q,k \in \N$ satisfying $k \leq q$
$$P^\star(G,q,k) = \frac{P(G \mathbin{\square} K_k, q)}{k!}.$$
\end{lem}

\begin{proof}
The result is clear when $q < \chi(G)$.  So, assume that $q \geq \chi(G)$.  Suppose $L$ is the $q$-assignment for $G$ that assigns $[q]$ to every vertex in $G$.  Let $A$ be the set of all proper $L$-packings of size $k$ of $G$.  Clearly, $|A|= P^\star(G,q,k)$, and $A$ is nonempty by Observation~\ref{obs: key} since $q \geq \chi(G) = \chi(G \mathbin{\square} K_k)$. For each $P \in A$, let $\pi_P$ be the set of bijections from $[k]$ to $P$.  Then, let $D = \bigcup_{P \in A} (\{P\} \times \pi_P)$, and let $C$ be the set of all proper $q$-colorings of $H=G \mathbin{\square} K_k$.

Now, let $f: D \rightarrow C$ be given by $f(P, \sigma) = c$ where $c$ is the proper $q$-coloring of $H$ given by $c(v_j,w_i) = \sigma(i)(v_j)$ for each $i \in [k]$ and $j \in [n]$.  Also, let $g : C \rightarrow D$ be given by $g(c) = (P, \sigma)$ where $f_i : V(G) \rightarrow [q]$ is given by $f_i(v) = c(v,w_i)$ for each $i \in [k]$, $P = \{f_i : i \in [k]\}$, and $\sigma$ is the element of $\pi_P$ given by $\sigma(i) = f_i$.  Since $f$ and $g$ are inverses of each other, $f$ is a bijection which means $|A| = |C|/k!$.  The desired result immediately follows.
\end{proof}

We immediately get the following result from Lemma~\ref{lem: connect}.

\begin{cor} \label{pro: complete}
Suppose $n \in \N$ and $G=K_n$.  Then, $P^\star(G,q,k) = 0$ whenever $k \leq q \leq n$, and $P^\star(G,q,k) = \frac{L(n,k,q)}{k!}$ whenever $q \geq n$ and $k \leq q$ where $L(n,k,q)$ denotes the number of $n \times k$ Latin arrays containing at most $q$ symbols. 
\end{cor}

Note that when $n=2$ and $k=q$, we get $P^\star(K_2,q) = P^\star(K_2,q,q) = \frac{L(2,q,q)}{q!} = \frac{(q!)(!q)}{q!} = !q$ which agrees with Proposition~\ref{pro: easytree}.  

We are now ready to prove a list version of Lemma~\ref{lem: connect}.

\begin{lem} \label{lem: connectlist}
Suppose $G$ is a graph.  Suppose $L$ is a $q$-assignment for $G$ and $k \in \N$ satisfies $k \leq q$.  Then,
$$P^\star(G,L,k) = \frac{P(G \mathbin{\square} K_k, L^{(k)})}{k!} \geq \frac{P_{\ell}(G \mathbin{\square} K_k , q)}{k!}.$$
\end{lem}

\begin{proof}
Let $H= G \mathbin{\square} K_k$ and $A$ be the set of all proper $L$-packings of size $k$ of $G$.  By observation~\ref{obs: key}, notice $A$ is empty if and only if there are not any proper $L^{(k)}$-colorings of $H$.  So, we may assume that $A$ is nonempty.  Clearly, $|A|= P^\star(G,L,k)$. For each $P \in A$, let $\pi_P$ be the set of bijections from $[k]$ to $P$.  Then, let $D = \bigcup_{P \in A} (\{P\} \times \pi_P)$, and let $C$ be the set of all proper $L^{(k)}$-colorings of $H$.

Now, let $f: D \rightarrow C$ be given by $f(P, \sigma) = c$ where $c$ is the proper $L^{(k)}$-coloring of $H$ given by $c(v_j,w_i) = \sigma(i)(v_j)$ for each $i \in [k]$ and $j \in [n]$.  Also, let $g : C \rightarrow D$ be given by $g(c) = (P, \sigma)$ where $f_i : V(G) \rightarrow [q]$ is given by $f_i(v) = c(v,w_i)$ for each $i \in [k]$, $P = \{f_i : i \in [k]\}$, and $\sigma$ is the element of $\pi_P$ given by $\sigma(i) = f_i$.  Since $f$ and $g$ are inverses of each other, $f$ is a bijection which means $|A| = |C|/k!$.  The desired result immediately follows.
\end{proof}

We can now put all these ingredients together to prove Theorem~\ref{thm: equal}.

\begin{proof}
Let $H = G \mathbin{\square} K_k$.  Notice that $q \geq nk(k-1)/2 + mk - 1 = |E(H)|-1$.  So, Theorem~\ref{thm: DZ} implies that $P_{\ell}(G \mathbin{\square} K_k , q) = P(G \mathbin{\square} K_k , q).$  Now, combining Lemmas~\ref{lem: connect} and~\ref{lem: connectlist}, we obtain:
$$P^\star(G,q,k) \geq P_{\ell}^\star(G,q,k) \geq \frac{P_{\ell}(G \mathbin{\square} K_k , q)}{k!} = \frac{P(G \mathbin{\square} K_k, q)}{k!} = P^\star(G,q,k).$$
The result immediately follows. 
\end{proof}

\section{Exponential Lower Bounds} \label{exp}

The strategy in this section is to follow the framework established in Section~\ref{Cartesian} and then use the ideas established in~\cite{BG22,DK23} to generalize the bounds on the list packing number to their enumerative counterparts, leading to exponential lower bounds on the corresponding list packing functions. Specifically, we wish to use Lemma~\ref{lem: connectlist} in conjunction with a slightly simplified version of a well-known result of Alon and F\"{u}redi~\cite{AF93} on the number of non-zeros of a polynomial.

\begin{thm} [B. Bosek, J. Grytczuk, G. Gutowski, O. Serra, M. Zajac~\cite{BG22}] \label{thm: bound}
Let $\mathbb{F}$ be an arbitrary field, let $A_1$, $A_2$, $\ldots$, $A_n$ be any non-empty subsets of $\mathbb{F}$, and let $B = \prod_{i=1}^n A_i$.  Suppose that $P \in \mathbb{F}[x_1, \ldots, x_n]$ is a polynomial of degree $d$ that does not vanish on all of $B$.  If $S = \sum_{i=1}^n |A_i|$, $t = \max |A_i|$, $S \geq n + d$, and $t \geq 2$, then the number of points in $B$ for which $P$ has a non-zero value is at least $t^{(S-n-d)/(t-1)}.$
\end{thm} 

We now prove our main lower bound on the number of proper $L$-packings of a sparse $G$.

\begin{customlem}{\bf \ref{lem: together}}
Suppose $G$ is an $n$-vertex graph with $m$ edges.  Suppose $L$ is a $q$-assignment for $G$, and $k \in \N$ is such that $k \leq q$ and $P^\star(G,L,k) > 0$.  If $m \leq n(q-1-(k-1)/2)$, then 
$$P^{\star}(G,L,k) \geq \frac{1}{k!} q^{kn-\frac{nk(k-1)/2 + km}{q-1}}.$$
\end{customlem}

\begin{proof}
We will prove $P(G \mathbin{\square} K_k, L^{(k)}) \geq q^{kn-\frac{nk(k-1)/2 + km}{q-1}}$ which will imply the desired result by Lemma~\ref{lem: connectlist}. 
 Let $H = G \mathbin{\square} K_k$.  Suppose that $L$ is such that $L(v_i) \subset \R$ for each $i \in [n]$.  Now, suppose $f$ is the $kn$-variable polynomial with real coefficients and variables $x_{i,j}$ for each $(i,j) \in [n] \times [k]$ given by
$$f = \prod_{(v_q,w_r)(v_s,w_u) \in E(H),\; r < u \text{ or } q < s \text{ when }r=u} (x_{q,r}-x_{s,u}).$$
Clearly, $f$ is of degree $|E(H)| = nk(k-1)/2 + mk$.  For each $(i,j) \in [n] \times [k]$, let $A_{i,j} = L^{(k)}(v_i,w_j)$, and let $B = \prod_{(i,j) \in [n] \times [k]} A_{i,j}$.  By the formula for $f$, we have that for any proper $L^{(k)}$-coloring for $H$, $g$, inputting $g(v_i,w_j)$ for $x_{i,j}$ for each $(i,j) \in [n] \times [k]$ results in a nonzero output for $f$. 

Consequently, $P(G \mathbin{\square} K_k, L^{(k)})$ is the number of elements in $B$ for which $f$ has a nonzero value.  By Observation~\ref{obs: key}, $P^\star(G,L,k) > 0$ implies that $f$ does not vanish on all of $B$.  Finally, Theorem~\ref{thm: bound} yields the desired result. 
\end{proof}

It should be noted that Lemma~\ref{lem: together} remains true if the bound on $m$ is dropped.  This is because when the bound on $m$ is violated, Lemma~\ref{lem: together} yields $P^\star(G,L,k) \geq 1$.  We however include the bound on $m$ in the statement since Lemma~\ref{lem: together} can only be useful when this bound is satisfied.  

We are now ready to prove Corollary~\ref{cor: 2colorings}.  Recall that if $G$ is a planar graph with girth at least $g$, then $|E(G)| \leq g|V(G)|/(g-2)$. The following result was recently obtained.

\begin{thm} [\cite{CC23}] \label{thm: recent2}
Suppose that $G$ is a planar graph of girth at least 8.  Then, $P_{\ell}^\star(G,3,2)>0$.
\end{thm}

It is shown in \cite{CC23} this result is best possible in the sense that $P_{\ell}^\star(G,3,3)=0$.  We can now combine Lemma~\ref{lem: together} and Theorem~\ref{thm: recent2} to show that there are exponentially many pairs of disjoint $L$-colorings for any 3-assignment $L$ of a planar graph of girth at least 8.

\begin{customcor} {\bf \ref{cor: 2colorings}}
Suppose $G$ is an $n$-vertex planar graph of girth at least 8.  Then,
$$P_{\ell}^\star(G,3,2) \geq \frac{3^{n/6}}{2}.$$
\end{customcor}

\begin{proof}
 Suppose $L$ is a 3-assignment for $G$ satisfying $P^\star(G,L,2)= P_{\ell}^\star(G,3,2)$. 
 Theorem~\ref{thm: recent2} implies $P^\star(G,L,2) > 0$.  Then, Lemma~\ref{lem: together} along with the fact that $|E(G)| \leq 4n/3$ yields
 $$P_{\ell}^\star(G,3,2) = P^\star(G,L,2) \geq \frac{3^{n/6}}{2}.$$
\end{proof}
\vspace*{0.5cm}
{\bf Acknowledgement.} The authors thank S. Cambie and W. Cames van Batenburg for their helpful comments.


\begin{thebibliography}{99}
{\small
\bibitem{AF93} N. Alon and Z. F\"{u}redi, Covering the cube by affine hyperplanes, \emph{European J. Combin.} 14 (1993), 79-83.

\bibitem{B94} N. Biggs, (1994) \emph{Algebraic graph theory}.  New York, NY: Cambridge University Press.

\bibitem{B12} G. D. Birkhoff, A determinant formula for the number of ways of coloring a map, \emph{The Annals of Mathematics} 14 (1912), 42-46.

\bibitem{BG22} B. Bosek, J. Grytczuk, G. Gutowski, O. Serra, and M. Zajac, Graph polynomials and group coloring of graphs, \emph{European J. Combin.} 102 (2022), 103505.

\bibitem{CC21} S. Cambie, W. Cames van Batenburg, E. Davies, and R. J. Kang, Packing list-colourings, \emph{Random Structures \& Algorithms} 64 (2024), 62-93.

\bibitem{CC23} S. Cambie, W. Cames van Batenburg, and X. Zhu, Disjoint list-colorings for planar graphs, arXiv: 2312.17233 (preprint), 2023. 

%\bibitem{CC24}  S. Cambie and W. Cames van Batenburg, \emph{personal communication}, 2024.

\bibitem{CS24} D. Cranston and E. Smith-Roberge, List packing and correspondence packing of planar graphs, \emph{Journal of Graph Theory}, \emph{to appear}.

\bibitem{DKM22} S. Dahlberg, H. Kaul, and J. Mudrock, An algebraic approach for counting DP-3-colorings of sparse graphs, \emph{European J. Combin.} 118 (2024), 103890.

\bibitem{DK23} S. Dahlberg, H. Kaul, and J. A. Mudrock, A polynomial method for counting colorings of sparse graphs, arXiv:2312.11744 (preprint), 2024. 

\bibitem{D00} F. Dong, Proof of a chromatic polynomial conjecture, \emph{Journal of Combinatorial Theory Series B} 78(1) (2000), 35-44.

\bibitem{DKT05} F. Dong, K. M. Koh, and K. L. Teo, \emph{Chromatic polynomials and chromaticity of graphs}, World Scientific, 2005.

\bibitem{DZ22} F. Dong and M. Zhang, How large can $P(G,L)-P(G,k)$ be for $k$-assignments $L$,  \emph{Journal of Combinatorial Theory Series B} 161 (2023), 109-119.

\bibitem{D92} Q. Donner, On the number of list-colorings, \emph{J. Graph Theory} 16 (1992), 239-245.

\bibitem{ET79} P. Erd\H{o}s, A. L. Rubin, and H. Taylor, Choosability in graphs, \emph{Congressus Numerantium} 26 (1979), 125-127.

\bibitem{G95} F. Galvin, The list chromatic index of a bipartite multigraph, \emph{J. Combinatorial Theory Series B} 63 (1995), no. 1, 153-158.

\bibitem{KS90} A. V. Kostochka and A. Sidorenko, Problem Session of the Prachatice Conference on Graph Theory, \emph{Fourth Czechoslovak Symposium on Combinatorics, Graphs and Complexity}, Ann. Discrete Math. 51 (1992), 380.

\bibitem{M23} J. A. Mudrock, A short proof that the list packing number of any graph is well defined. \emph{Discrete Math.} 346(11) (2023), 113185.

\bibitem{T14} T. Tao, Algebraic combinatorial geometry: the polynomial method in arithmetic combinatorics, incidence combinatorics, and number theory, \emph{EMS Surv. Math. Sci.} 1 (2014), 1-46.

\bibitem{T94} C. Thomassen, Every planar graph is 5-choosable, \emph{J. Combin. Theory Ser. B} 62 (1994), 180-181.

\bibitem{T09} C. Thomassen, The chromatic polynomial and list colorings, \emph{J. Combin. Theory, Ser. B} 99 (2009), 474-479.

\bibitem{V76} V. G. Vizing, Coloring the vertices of a graph in prescribed colors, \emph{Diskret. Analiz.} no. 29, \emph{Metody Diskret. Anal. v Teorii Kodovi Skhem} 101 (1976), 3-10.

\bibitem{WQ17} W. Wang, J. Qian, and Z. Yan, When does the list-coloring function of a graph equal its chromatic polynomial, \emph{Journal of Combinatorial Theory, Series B} 122 (2017) 543-549.

\bibitem{W01} D. B. West, (2001) \emph{Introduction to Graph Theory}.  Upper Saddle River, NJ: Prentice Hall.
}

\end{thebibliography}
\end{document}